\documentclass[11pt]{article} 
\usepackage{amsmath, amssymb} 
\usepackage{amsthm}
\usepackage{graphicx}
\usepackage{graphics}

\def\om{\omega}
\def\Om{\Omega}

\def\p{  \partial  }

\def\demi{\frac{1}{2}}

\def\L2Om{L^2(\Om)}

\def\L2om{L^2(\om)}

\theoremstyle{plain}

\newtheorem{proposition}{Proposition}[section]
\newtheorem{definition}{Definition}[section]

\newtheorem{lemma}{Lemma}[section]
\theoremstyle{remark}
\newtheorem{remark}{Remark}[section]

\numberwithin{equation}{section}

\newcommand{\PBS}[1]{\let\temppp=\\#1\let\\=\temppp}

\title{
Periodic solutions of o.d.e. systems with a lipchitz non linearity 
}
\author{ Bernard Rousselet\thanks{ Universit\'{e} de Nice Sophia Antioplis ,  J.A.D. laboratory of mathematics and interactions, UMR CNRS 6621, Parc Valrose, 06108 Nice, France; br unice.fr} }
\date{\today}

\begin{document}
\bibliographystyle{alpha}
\maketitle

\begin{abstract}
  In this report, we address differential systems  with Lipschitz non linearities; this study is motivated by the subject of vibrations of structures with unilateral springs or non linear stress-strain law close to the linear case.
We consider  existence and solution with fixed point methods; this method is constructive and provides a numerical algorithm which is under study.
We describe the method for a static case example and we address  periodic solutions of differential systems arising in the vibration of structures.
\end{abstract}
 \section{Introduction}
 \label{sec:intro}
  
This work stems from non destructive testing experiments reported in \cite{vdb-lagier04} and from the experiments and the computations of a beam model of the dynamic of satellite solar arrays\footnote{ with the support of Thales Alenia Space (Cannes, France)}  \cite{hazim-ecssmt},\cite{hazim-these}. In these two situations the structure is submitted to an harmonic force; the behavior of the structure is locally non linear in particular due to unilateral elastic contacts.

 During these studies, it became clear that an in-depth understanding of the forced dynamic behavior relies on the study of periodic solution of the free system; this can be considered as an extension of the use of normal modes of a linear  free system in order to study the dynamics of the forced associated linear system as explained in \cite{Geradin-rixen} (in English see \cite{Geradin-rixen-eng}). Periodic solutions for non linear o.d.e. arising in structural dynamics have been considered in \cite{a3,bellizi-gkm01,sj-brgdr08,nnm-kpgv, jiang-pierre-shaw04, bel2, mikhlin10} and in many others.

In \cite{junca-br10}, the Lindstedt-Poincarr\'e method was used in order to derive approximate non linear normal modes (a periodic solution close to a linear normal mode) for small non linearity.

At this point it is worth to recall  a fundamental remark of Henri Poincarr\'e in his report "Sur le probl\`eme des trois corps et les \'equations de la dynamique":
\begin{quote}
``Ce qui rend ces solutions p\'eriodiques aussi pr\'ecieuses, c'est qu'elle sont, pour ainsi dire, la seule br\`eche  par o\`u nous puissions p\'en\'etrer dans une place jusqu'ici r\'eput\'ee inabordable'' (\cite{poincarre92-99}). 
\end{quote}
In English
\begin{quote}
  ``What renders these periodic solutions so precious is that they are, so to speak, the only breach through which we may try to penetrate a stronghold previously reputed to be impregnable''
\end{quote}
Although this remark is mainly devoted to celestial mechanic, it is also meaning-full for  vibration of structures.

  In this report we study periodic solutions of non linear systems of ODE arising in free vibrations of structures subjected to unilateral springs; the ideas and the methods may be used for systems modeling other situations in particular when the non linearity is only Lipschitz.
We recall a simple modeling process in section \ref{sec:eq-structmech}.
Then in section \ref{sec:fromnewton}, we start by presenting some ideas of quasi Newton type algorithm for solving equilibrium of structures in the static case with some Lipschitz non linearities; for example, these non linearities arise in cases where structures are subjected to unilateral springs; these unilateral springs are often simplified models of structures like bumpers usually made of viscoelastic materials like in \cite{hazim-ecssmt}; the same model will be addressed in the dynamic case.
We derive a constructive proof of existence of periodic solutions for a one degree of freedom system in section \ref{sec:onedof}; then the case of several degrees of freedom is addressed in section \ref{sec:vibrat-several}.

Vibrations with non linearities have been considered from an experimental point of view in \cite{vdb-lagier-groby,vdb-lagier04,land_mine02,moussatov-castagnede-gusev,zaitsev_fillinger_gusev_c05}; asymptotic expansions and numerical methods have been used in \cite{rousselet-vdb03, brvdbfeijo06,rousselet-vdb-inria04,rousseletvdb05,vdbbr_saviac05,brvdbfeijo06,br-hammamet08,br-AFPAC,sj-brgdr08,benbrahim-tamtam09, benbrahimSmai,junca-br10,benbrahimGdrafpac}.

 The work of Lyapunov \cite{lyapunov49} is often cited as a basis for the existence of periodic solutions which tends towards linear normal modes as amplitudes tend to zero; this result uses the {\it hypothesis of analycity} of the function involved in the differential system. As we address the case where the non linearity is only Lipschitz, this hypothesis of analyticity is obviously violated.

Some general existence results, based on calculus of variations, for convex Hamiltonian systems are presented in \cite{ekeland-convhamilton}; the case of analytic functions in connexion with the use of normal forms is considered in \cite{jezequel-lamarque91, touze-amabili06, Iooss3}.
Energy pumping has been addressed in several papers, for example in \cite{schmidt-lamarque10}. 


 The case of vibration of structures with unilateral springs  or more generally systems modeled with Lipschitz non linearities are considered in some recent studies (experiments, asymptotic expansions, numerical computations): \cite{hazimSmai,hr-brgdr08, hazim-tamtam09, hazim-br08, hazim-ecssmt,hazim-these,sj-brgdr08, junca-br10, jiang-pierre-shaw04, vestroni08}. A study of dynamics of elastic shocks is in \cite{attouch-cabot-redont}.
The case of distributed systems modeled by partial differential equations has been addressed for example in \cite{ps4};  general references for perturbation methods are among many others  \cite{sanders-verhulst,nayfeh81} for differential equations and \cite{KC68} for partial differential equations.

A review of periodic solutions of non autonomous ordinary differential equations may be found in \cite{DBLP:Mawhin09} for details, see also \cite{rm2-73}.

In the static case, for example, a stability result concerning the obstacle problem for a plate is considered in \cite{leger1}.

Non smooth optimization is an active field of research; we only cite some references connected to this research:  \cite{dontchev-rocka}, \cite{kunisch08}, \cite{leine-nijmeijer04} and the references therein. 

Here we intend to use methods from non smooth optimization to prove existence of periodic solutions ans to derive a numerical algorithm to find them; the precise description and implementation of the algorithm will be addressed in a forthcoming paper.




\section{Equations from structural mechanics}
\label{sec:eq-structmech}
Consider a spring-mass system with some weak broken springs which are
acting only in compression.  Let $u $ denote, the displacement of the
masses, and  $\gamma_j$ the strain of the springs; it  is related to the displacement by
$\gamma_j=u_{j+1}-u_j$ or in vector form with an incident matrix $B$
$$\gamma=B u$$
if we assume that the material is elastic linear, the stress-strain law is
$$\sigma=E \gamma + \epsilon E'( \gamma +d)_-$$
where $\epsilon$ is a small parameter; $d$ denotes some backlash; in other words for some springs in small compression or in traction, there is no induced stress.

In the static case, the force applied to each mass is denoted by $Y$ and we have the equilibrium equation
$$B^T \sigma=Y$$ 
or more explicitely $B^T E \gamma + \epsilon B^T E'( \gamma +d)_-=Y$ or
\begin{equation}
  \label{eq:Ku+eps=Y}
  B^T E B u + \epsilon B^T E'( Bu +d)_-=Y   
\end{equation}

  \begin{remark}
  \begin{itemize}
  \item 
Obviously this system is quite general; many situations of structural mechanics may be cast in this system by using finite elements. Moreover many other physical systems can be modeled with such non linearities.
  \item 
To solve this system, the idea is to use some quasi Newton method taylored to systems involving Lipschitz functions.
\item The proof of convergence relies on a fixed point method.

The principle of the  method may be found in the book of Dontchev-Rockafellar \cite{dontchev-rocka} as a way of proving the existence of an inverse of a Lipschitz function $f$ or an implicit theorem for Lipschitz functions.
\item 
Here the function $f(X)$ is the left hand side and we want to solve $f(X)=Y$ for some values of $Y$;
note that here $f$ is only a Lipschitz function; it is not continuously differentiable!
  \end{itemize}
  \end{remark}

\paragraph{ Example with 5 masses and 4 springs on a straight line} 

The incident matrix 
is
\begin{displaymath}
  B=
\left [ \begin{array}{ccccc}
    1 & -1 &0 &0   & 0  \\
    0 &1   &-1&  0 & 0  \\
    0  & 0  & 1&-1  & 0  \\
    0 &0   & 0& 1  &-1
  \end{array} \right ]
\end{displaymath}
As we have enforced no boundary condition, there is a rigid body movement: $u_j= c$ for $j=1, \dots 5$ gives the strain $\gamma=0$; for example, we can enforce $u_0=0$ (the mass 0 is attached) and remove it from the degrees of freedom; so we get
\par
\begin{displaymath}
  B=
\left [ \begin{array}{cccc}
      -1 &0 &0   & 0  \\
     1   &-1&  0 & 0  \\
       0  & 1&-1  & 0  \\
     0   & 0& 1  &-1
  \end{array} \right ]
\end{displaymath}
In the case where the only  broken spring is the second:
\begin{displaymath}
 B^T E'( Bu +d)_-= 
\left [
  \begin{array}{c}
  -E'_2(u_3-u_2 +d)_- \\
   E'_2(u_3-u_2 +d)_- \\
   0\\
   0
  \end{array} \right ]
\end{displaymath}

\paragraph{Example with weak broken springs at supports on a straight line} \label{par:ex-broken-supp}
\setlength{\unitlength}{1cm}
\begin{figure} 
  \centering
\begin{picture}(21,4)
\thicklines
\put(-.3,0){\framebox(.4,2){}}
\thinlines
\put(.3,1.2){\line(1,-2){0.2}} 
\multiput(.5,.8)(.4,0){4}{\line(1,4){.2}}
\multiput(.7,1.6)(.4,0){3}{\line(1,-4){0.2}} 
\put(1.9,1.6){\line(1,-2){0.2}} 
\put(2.1,1.2){\line(1,0){0.4}} 
\put(2.4,1.2){\circle*{.4}}

\put(2.6,1.2){\line(1,0){0.3}} 
\put(2.9,1.2){\line(1,-2){0.2}} 
\multiput(3.1,.8)(.4,0){4}{\line(1,4){.2}}
\multiput(3.3,1.6)(.4,0){3}{\line(1,-4){0.2}} 
\put(4.5,1.6){\line(1,-2){0.2}} 
\put(4.7,1.2){\line(1,0){0.4}} 
\put(5.,1.2){\circle*{.4}}

\put(5.2,1.2){\line(1,0){0.3}} 
\put(5.5,1.2){\line(1,-2){0.2}} 
\multiput(5.7,.8)(.4,0){4}{\line(1,4){.2}}
\multiput(5.9,1.6)(.4,0){3}{\line(1,-4){0.2}} 
\put(7.1,1.6){\line(1,-2){0.2}} 
\put(7.3,1.2){\line(1,0){0.4}} 
\put(7.6,1.2){\circle*{.4}}

\put(7.8,1.2){\line(1,0){0.3}} 
\put(8.1,1.2){\line(1,-2){0.2}} 
\multiput(8.3,.8)(.4,0){4}{\line(1,4){.2}}
\multiput(8.5,1.6)(.4,0){3}{\line(1,-4){0.2}} 
\put(9.7,1.6){\line(1,-2){0.2}} 
\put(9.9,1.2){\line(1,0){0.4}} 
\put(10.2,1.2){\circle*{.4}}

\put(10.4,1.2){\line(1,0){0.3}} 
\put(10.7,1.2){\line(1,-2){0.2}} 
\multiput(10.9,.8)(.4,0){4}{\line(1,4){.2}}
\multiput(11.1,1.6)(.4,0){3}{\line(1,-4){0.2}} 
\put(12.3,1.6){\line(1,-2){0.2}} 
\put(12.5,1.2){\line(1,0){0.4}} 
\put(12.8,1.2){\circle*{.4}}

\put(13,1.2){\line(1,0){0.3}} 
\put(13.3,1.2){\line(1,-2){0.2}} 
\multiput(13.5,.8)(.4,0){4}{\line(1,4){.2}}
\multiput(13.7,1.6)(.4,0){3}{\line(1,-4){0.2}} 
\put(14.9,1.6){\line(1,-2){0.2}} 
\thicklines
\put(15.3,0){\framebox(.3,2){}}
\end{picture}
\label{fig:5mass}
\caption{5 masses and 6 springs broken at supports}
\end{figure}
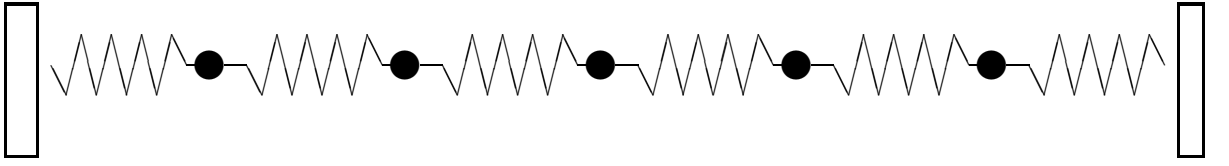

We consider the case where the springs are broken at both ends; then for $ n$ masses, we have $n+1$ springs;
\begin{gather}
  \gamma_1=u_1 , \; \gamma_j=u_j-u_{j-1}, \; \gamma_{n+1}=-u_n
\end{gather}
For $n=5$, we get (see figure \ref{fig:5mass})

\begin{displaymath}
  B=
\left [ \begin{array}{ccccc}
    1 & 0 &0 &0   & 0  \\
    -1 &1   & 0&  0 & 0  \\
    0  & -1  & 1& 0 & 0  \\
    0 &0   & -1& 1  & 0 \\
   0 &0   & 0 &-1 & 1 \\
    0 &0   & 0 & 0&-1
  \end{array} \right ]
\end{displaymath}
In the case of broken springs at both ends, 
\begin{gather}
  \sigma_1=E'_1 (\gamma_{1}+d_1)_-, \; \sigma_j= E_j \gamma_j, \; \text{ for } j=2 \dots 5, \; \sigma_6=E'_6 (\gamma_{6}+d_6)_-
\end{gather}
so that, the equilibrium equation are:
\begin{gather}
  K u + B^TE'(Bu+d)_{-} =F \quad 
\end{gather}
with
\begin{gather}
  K=B^T E B, \quad \text{ the rigidity matrix, and} \\
[B^TE'(Bu+d)_{-}]_1=E'_1(u_1+d_1)_- \\
[B^TE'(Bu+d)_{-}]_j=0, \quad \text{ for } \quad j=2, \dots 4, \quad
[B^TE'(Bu+d)_{-}]_5=E'_6(-u_5+d_5)_- 
\end{gather}
\begin{displaymath}
K=
\begin{bmatrix}
  E_2 & -E_2 &0&0&0 \\
  -E_2 & E_2+E_3 & -E_2 &0&0\\
  0 &-E_3 & E_3+E_4 & -E_4 &0\\
  0 &0& -E_4 & E_4+E_5 & -E_5  \\
  0&0&0&                   -E_5& E_5 
\end{bmatrix}
\end{displaymath}
\begin{remark}
  \begin{enumerate}
  \item We notice that $K$ is not invertible; see Remark \ref{rem:natural}
  \item Obviously, we can easily extend this example to a bi or tridimensional array of masses and springs, with unilateral springs or to non linear springs with lipschitz non linearity..
  \end{enumerate}

\end{remark}

  \section{From a classical Newton method to quasi Newton method with a Lipschitz function}
\label{sec:fromnewton}
In this introductory section, we recall a constructive convergence proof of the Newton-Raphson algorithm based on the fixed point method; then we extend the approach to a quasi Newton method for a smooth function; finally, we extend it to a sum of a smooth and a lipschitz function.
  
\subsection{Newton and Quasi Newton methods in the smooth case}

\subsubsection{Newton method}
When $f$ is differentiable, in order to solve an equation $f(x)=y $,  the idea is to approximate $f$ by its differential:
we approach $f(x_{k+1})$ by $f(x_k) +f'(x_k)(x_{k+1} -x_k)$ so that 
$x_{k+1}$ is the solution of:
$$f(x_k) +f'(x_k)(x_{k+1} -x_k) =y$$ \label{eqref:newtonk}

To prove the convergence, we consider the equation for the previous index
$$f(x_{k-1}) +f'(x_{k-1})(x_{k} -x_{k-1}) =y$$
and manipulating, we obtain
$$ f'(x_k)(x_{k+1} -x_k)=-[f(x_k) - f(x_{k-1}) -f'(x_{k-1})(x_{k} -x_{k-1}) ] $$
or 
$$(x_{k+1} -x_k)=- [f'(x_k)]^{-1} \;[f(x_k) - f(x_{k-1}) -f'(x_{k-1})(x_{k} -x_{k-1}) ] $$
Now, it is clear that to show convergence of this sequence, we need mainly the 

{\bf Assumptions}
\begin{itemize}
\item 

$\| f'(x)^{-1}\| \le c $ \label{eqref:f'm1c} in a neighborhood of the solution;

\item 
and an easy consequence of differentiability of $f$ in a neighborhood of $\bar{x}$ 

for any $\epsilon >0 $, there exists $\delta$ such that for $\|x_k-\bar{x}\| \le \delta, \;\|x_{k+1}-\bar{x}\| \le \delta $
$$\|f(x_k) - f(x_{k-1}) -f'(x_{k-1})(x_{k} -x_{k-1}) \| < \epsilon \|x_{k} -x_{k-1} \|$$
\end{itemize}
we choose $\epsilon=\frac{1}{2c}$ and we obtain

$$ \|(x_{k+1} -x_k)\| \le \demi  \|(x_{k} -x_{k-1})\| $$
  \begin{itemize}
  \item 

which proves convergence (if the initial data is close enough to the solution itself close enough to the given $\bar{x}$).
\item 

In this method, it is well known that the main trouble is the inequality  
 $$\| f'(x)^{-1}\| \le c$$
 if $c$ is  large, we have to start very close to the solution in order to converge;
 the ``tangent subspace'' to the graph of $f$ near the solution should not be ``too horizontal''.
A way to circumvent this trouble is to embed the Newton method in a continuation process.
\end{itemize}
\subsubsection{Quasi Newton method}
The idea of a quasi Newton method is to evaluate the derivative of $f$ at a fixed point $\bar{x}$.
Still considering the case where $f$ is differentiable, instead of 
\ref{eqref:newtonk}, we use
$$f(x_k) +f'(\bar{x})(x_{k+1} -x_k) =y$$
in which the derivative of $f$  is evaluated at the fixed point $\bar{x}$.
To prove convergence, we consider the same equation for the previous index
$$f(x_{k-1}) +f'(\bar{x})(x_{k} -x_{k-1}) =y$$
manipulating, we get:

$$ f'(\bar{x})(x_{k+1} -x_k)=-[f(x_k) - f(x_{k-1}) -f'(\bar{x})(x_{k} -x_{k-1}) ] $$
We use the same assumption on the inverse of the derivative $f'$ and still with differentiability in the neighborhood of $\bar{x}$, we get
for any $\epsilon > 0$
$$\|f(x_k) - f(x_{k-1}) -f'(\bar{x})(x_{k} -x_{k-1}) \| < \epsilon \|x_{k} -x_{k-1} \|$$
then, from
$$(x_{k+1} -x_k)=- [f'(\bar{x})]^{-1} \;[f(x_k) - f(x_{k-1}) -f'(x_{k-1})(x_{k} -x_{k-1}) ] $$
with the same assumptions, we still get
$$ \|(x_{k+1} -x_k)\| \le \demi  \|(x_{k} -x_{k-1})\| $$
which proves convergence (if the initial data is close enough to the solution itself close enough to the given $\bar{x}$).

\subsection{Quasi Newton method in a Lipschitz case}
\label{subsec:quasi-lip}
We use some notions introduced in Dontchev-Rockafellar \cite{dontchev-rocka}.

 We no longer assume that $f$ is differentiable but that there exists   a {\bf strict estimator } $h$ of $f$ at $\bar{x}$: 
 \begin{definition}
   \begin{enumerate}
   \item 
We recall that  the Lipschitz modulus of $e$ at $\bar{x}$ is
$$lip(e,\bar{x})= \limsup_{x,x' \rightarrow \bar{x}, x \ne x'} \frac{\|e(x)-e(x') \|}{\|x-x'\|} $$
\item 
$h$ is a strict estimator of $f$ at $\bar{x}$, when
   $$f(x)= h(x) + e(x)$$ 
with the Lipschitz modulus of $e$ satisfying $$lip(e,\bar{x}) \le \mu < +\infty$$ 
   \end{enumerate}
 \end{definition}

We assume that 
$$ \bar{y}= f(\bar{x})$$

Moreover, we assume that $h$ is invertible around $\bar{y}$ for $\bar{x}$

 and 
 \begin{equation}
   \label{eq:liph-1}
lip(h^{-1}, \bar{y}) \le \kappa  
 \end{equation}


 \begin{remark}
   It should be pointed out that it stems directly from the previous definition that the assumption, the Lipschitz modulus at $\bar x$ of a function  be finite,  implies that this function is Lipschitz in a neighborhood of $\bar x$.
 \end{remark}
\subsubsection{ Examples}
\label{subsub:examples}
\begin{itemize}
\item $f(x)=(x)_-$, $lip(f, x)= 1$ for $x \le 0$ but  $lip(f, x)= 0$ for $x > 0$
\item $e(x)= x_- - \alpha x$ with $0 \le \alpha \le 1$
$$lip(e,x) \le \max(\alpha, 1- \alpha)$$
\item $e=(Bu+ d)_- - \Lambda (Bu+d) $ where $\Lambda$ is a {\bf diagonal} matrix with $0 \le \lambda_j \le 1$.

set $d_j(u)=\sum B_{jk}u_k +d_j$

as $e_j(u)-e_j(u')=d_{j-}-\lambda_j d_j$, we get

$|e_j(u)-e_j(u')| \le \max(\Lambda_j,1-\lambda_j)|d_j(u)-d_j(u')| \le  \max(\lambda_j,1-\lambda_j) \|B\| \|u-u'\| $
$$lip(e,x) \le \max_j\max(\lambda_j, 1- \lambda_j)\|B\|$$

\end{itemize}

consider now the case with 
\begin{gather}
\label{eq:f(u)=}
f(u)=   B^T E B u + \epsilon B^T E' (Bu+ d)_- \;.
\end{gather}

It is the left hand side of our example from structural mechanics; as the first term is linear, the function $e$ is 
$e=\epsilon [(Bu+ d)_- - \Lambda (Bu+d)] $
and $$lip(e,x) \le \epsilon \max_j\max(\lambda_j, 1- \lambda_j) \|B\|$$
\begin{gather}
\label{eq:h(u)=}
h(u)=  B^T E B u + \epsilon B^T E'\Lambda (Bu+ d) 
\end{gather}

is invertible when the rigidity  matrix $K_{\epsilon}= B^T E B  + \epsilon B^T E' B$ is invertible and its inverse is:
$$h^{-1}(y)= K_{\epsilon}^{-1} [y -  B^T E' \Lambda d] $$
and $$lip(h^{-1}, y) \le \| K_{\epsilon}^{-1} \|$$

So when $K_{\epsilon}$ is invertible, for $\epsilon$ small enough, $$lip(e,x) lip(h^{-1}, y) < 1$$
 this property is crucial for convergence but it should be pointed out that 
$K_0=B^T E B$ is not required to be invertible.

\subsubsection{Algorithm}

Now the idea is to build a sequence  to solve
$$ f(x)=y$$ by using the strict estimator $h$:
$$ h(x_{k+1}) +e(x_{k+1})= y$$
to compute $x_{k+1}$ solution of this equation  is as difficult as the initial problem but it may be approximated by replacing $e(x_{k+1})$ by   $e(x_{k})$
$$ h(x_{k+1}) +e(x_{k})= y$$
to prove convergence, we use the formula at the previous iteration
\begin{equation}
  \label{eq:hkek-1}
 h(x_{k}) +e(x_{k-1})= y
\end{equation}

we get
$$ x_{k+1}=  h^{-1}(y- e(x_k) ) \quad  x_{k}=  h^{-1}(y- e(x_{k-1}) ) $$
and
$$x_{k+1}=x_{k} +  h^{-1} (y- e(x_{k}) ) - h^{-1} (y-e(x_{k-1}) )$$
\label{eqref:xk+1-xk}

\begin{itemize}
\item 
with the {\bf assumption} \eqref{eq:liph-1}, we get that for $e(x_{k})$ and $e(x_{k-1})$ small enough and $y$ close enough to $\bar{y}$,
$$\| h^{-1}[ y- e(x_{k}) ] - h^{-1} [y-e(x_{k-1}) ] \le \lambda \|  - e(x_{k}) + e(x_{k-1}) \| $$
\item 
on the other hand, $h$ being a strict estimator:
$$\|   e(x_{k}) - e(x_{k-1}) \|  \le  \nu \|x_{k} - x_{k-1} \| $$
\item 
moreover, we assume that $\mu \; \kappa < 1$ so that for $x_{k}$ and  $ x_{k-1}$ close enough to $\bar{x}$, we may assume that $\lambda \; \nu < 1$ 
\item 
From \ref{eqref:xk+1-xk}, we get 
$$  \|x_{k+1} - x_{k} \|  \le \lambda  \;  \nu \|x_{k} - x_{k-1} \| $$
Usually, we have to show that  if $x_0$ is close enough to $\bar x$, $x_1$ lies in a ball in which $h^{-1}$ is Lipschitz; here it is obvious;
 it proves the geometric convergence when 
$$ \lambda  \;  \nu  < 1$$

\end{itemize}

We have obtained
\begin{proposition}
Consider the equation 
\begin{equation}
  \label{eq:fx=y}
 f(x)=y 
\end{equation}

  Assume that 
$$f(x)= h(x) + e(x)$$ with the Lipschitz modulus satisfying
$$lip(e,\bar{x}) \le \mu < +\infty
\quad lip(h^{-1}, \bar{y}) \le \kappa < +\infty $$ 
with $\mu \kappa <1$, $f(\bar x)=\bar y$ and $e(\bar x)=0$

then the sequence
$$ h(x_{k+1}) +e(x_{k})= y$$ with $x_0$ is given close to the solution,
converges geometrically to the solution of \eqref{eq:fx=y} close to $x_0$.
\end{proposition}

This is the case for our example from structural mechanics when $\epsilon$ is small enough.

\paragraph{Continuation process}

 This method may be embedded in a continuation process:
\begin{itemize}
\item 
Set a step size $\delta \epsilon$ and an initial value of $\epsilon_1$ and a starting point $\tilde{x_1}$:
\item For l from 1 to L
\item 
Use the non smooth quasi Newton method to find the solution $x_l^*$ of  the equation for $\epsilon_l $ with the starting point $\tilde{x}_l$ ;
\item increase $\epsilon_{l+1}=\epsilon_l + \delta \epsilon$ and set $\tilde{x}_{l+1}=x_l^*$
\item end For
\end{itemize}
Extensions for differential equations are considered below; control problems are possible extensions.

\paragraph{An example}
Consider the iterations \eqref{eq:hkek-1} for the example of broken springs at both ends \eqref{par:ex-broken-supp}, we get with \eqref{eq:f(u)=} and \eqref{eq:h(u)=}:
\begin{gather} \label{eq:iter-broken-s}
  Ku^{k+1}+\epsilon B^TE'\Lambda(Bu^{k+1}+d) =-\epsilon \left (B^TE'(Bu^{k}+d)_- - B^TE'\Lambda(Bu^{k}+d) )\right)+F
\end{gather}
We notice that the matrix of the left hand side is the sum of $K$ and $\epsilon B^TE'\Lambda B$; it is usually invertible.
\begin{remark} \label{rem:natural}
  Had we use the ``natural'' iterative method
  \begin{gather}
     Ku^{k+1} =-\epsilon B^TE'\Lambda(Bu^{k}+d) +F
  \end{gather}

we notice that the matrix $K$ is {\bf not } invertible so that this algorithm is not working at all; in cases where the smallest eigenvalue of $K$ is not zero but small as in problem arising in large aerospace structures, the algorithm will converge for $\epsilon$ quite smaller than with the previous algorithm \eqref{eq:iter-broken-s}

It could be proved that the matrix is better conditioned with the ``quasi Newton'' algorithm than with the natural algorithm.
\end{remark}

\section{Periodic solution of a one degree of freedom spring-mass  system with Lipschitz non linearity}
\label{sec:onedof}

\subsection{The differential equation, periodic solution}

We set 
\begin{equation}
  \dot{\tilde{x}}=\frac{d \tilde{x}}{dt}
\end{equation}
and we  consider a one d.o.f. spring-mass system with a linear  and a weak linear unilateral spring:
\begin{equation}
\label{eq:mx"+}
  m \ddot{\tilde{x}} + k \tilde{x} + \epsilon k (\tilde{x})_-=0 \text{ with } x_-=\frac{x-|x|}{2}
\end{equation}
or more generally a linear and a weak non linear spring
\begin{equation}
\label{eq:mx"+g}
  m \ddot{\tilde{x}} + k \tilde{x} + \epsilon k g(\tilde{x})=0
\end{equation}
We set $\omega^2=\frac{k}{m}$, so this equation may be written
\begin{equation} \label{eq:xomega2}
   \ddot{\tilde{x}} + \omega^2 \tilde{x} + \epsilon \omega^2 (\tilde{x})_-=0
\end{equation}
As we suspect that the frequency of a periodic solution depends on $\epsilon$, we perform the change of variable:
$$\theta=\omega_{\epsilon} t$$ with 
$$\frac{1}{\omega_{\epsilon}^2} =\frac{1-\epsilon \eta(\epsilon)}{\omega{^2}}$$
\begin{remark}
  This is a  classical change of variable, e.g. see  \cite{verhulst90}; it is used in the method of strained coordinates also called Linsted-Poincarr\'e; in particular, it has been used in \cite{junca-br10} to derive an asymptotic expansion for a quite similar system of differential equations. Instead, here we use it  to derive a { \it constructive } proof of existence of a periodic solution; we are only looking for periodic solutions whereas in \cite{junca-br10}, we can obtain a quasi periodic expansion.
\end{remark}

and we set
\begin{equation}
  x'=\frac{dx}{d\theta}
\end{equation}
If we set $x(\theta)=\tilde{x}(t)$, equation \eqref{eq:xomega2} may be written:
\begin{align}
  \label{eq:xeta}
 & x"+(1-\epsilon \eta)x+\epsilon(1-\epsilon \eta)x_-=0 \quad \text{ and in the general case}\\
 & x"+(1-\epsilon \eta)x+\epsilon(1-\epsilon \eta)g(x)=0
\end{align}
or
\begin{equation}
  \label{eq:xfeta}
  x"+x +\epsilon f(x,\eta,\epsilon)=0
\end{equation}
with
\begin{align}
  \label{eq:f=}
 & f(x,\eta,\epsilon)=-\eta x +(1-\epsilon \eta)x_- \quad
\text{ and in the general case }\\
  \label{eq:f=...g}
 & f(x,\eta,\epsilon)=-\eta x +(1-\epsilon \eta) g(x)
\end{align}
As the solution of \eqref{eq:xfeta} is:
\begin{equation}
  \label{eq:xintegrale}
  x=a cos(\theta) + b sin(\theta) -\epsilon \int_0^{\theta}sin(\theta-s)f(x(s),\eta,\epsilon)ds
\end{equation}
or
\begin{gather}
  x=a cos(\theta) + b sin(\theta) - \epsilon r(\theta, a,\epsilon,\eta,b)
\end{gather}
we obtain the following lemma which may be considered as an extension to the Lipschitz case of a classical result which, for example, may be found in \cite{verhulst90}.
\begin{lemma}
\label{lem:rdrF1F2}
Consider the solution $\tilde x$ of the Cauchy problem of equation \eqref{eq:mx"+} (resp. \eqref{eq:mx"+g} ) and the solution $x$ of the associated equation \eqref{eq:xfeta}, \eqref{eq:f=} (resp. \eqref{eq:f=...g}) after change of variable.
For $\epsilon$ close to zero, $\tilde{x}$ is a  solution of \eqref{eq:mx"+} (resp. \eqref{eq:mx"+g} ) of period $2\pi/\omega_{\epsilon}$ if and only if $x$ is a   solution of \eqref{eq:xfeta}, \eqref{eq:f=} (resp. \eqref{eq:f=...g}) and 
  \begin{equation}
\label{eq:F=0}
r(2 \pi, a,\epsilon,\eta,b)=0 \quad \text{ and } \quad \frac{\p r(2 \pi, a,\epsilon,\eta,b)}{\p \theta} =0 \quad \text{ or equivalently } \quad F(p,y)=0    
  \end{equation}

 with
\begin{equation}
  \label{eq:p,y=}
  p=(a,\epsilon), \; y=(\eta,b)
\end{equation}
\begin{align}
\label{eq:F1=}
  F_1(p,y) &=\int_0^{2\pi}sin(s) f(x,\eta,\epsilon) ds \\
\label{eq:F2=}
  F_2(p,y) &=\int_0^{2\pi}cos(s) f(x,\eta,\epsilon) ds 
\end{align}
\end{lemma}

\subsection{Computation of  $\eta(0)$}

\paragraph{Case $x_-$}
We have
$$ f(x,\eta,0)=-\eta x +x_-$$ and for $b=0$, 
$$f(x,\eta,0)=-\eta a cos(\theta)+(a cos(\theta))_-$$
 so 
\begin{align}
  F_1 &= \int_0^{2\pi} sin(s)(a cos(s)))_- ds\\
  F_2 &= \int_0^{2\pi} -\eta a cos(s)^2  +cos(s)(a cos(s)))_- ds
\end{align}
we note that  $F_1$ is identically zero; 
we remark that
\begin{gather}
  \text{if } a \ge 0 , \;(a cos(s))_-=a(cos(s))_- \text{ and if }  a \le 0 , \;(a cos(s))_-=a(cos(s))_+
\end{gather}
in both cases,  we get from $F_2$ 
$$\eta(0)=\frac{1}{2}$$ 
by using for example
\begin{gather}
   |\cos(s)| =  \frac{2}{\pi} -
      \frac{4}{\pi} \sum_{k=1}^{+\infty} \frac{(-1)^k}{4k^2-1}\cos(2ks),
\end{gather}
and by anticipating that $\eta$ is a Lipschitz function of $\epsilon$ (see proposition \ref{prop:periosol-1dof} )
$$\frac{1}{\omega_{\epsilon}^2 }=\frac{1-\epsilon/2+ o(\epsilon)}{\omega{^2}}$$
and
$$ x=a cos(\theta) + b sin(\theta) - \epsilon r(\theta, a,0,\eta(0),0) +o(\epsilon)$$
This result might be obtained by direct inspection of the level curves of the energy associated to the differential equation, see \cite{junca-br10}; but this approach may be extended to systems of differential equations.

\paragraph{General case}
For simplicity, we consider $b=0$, the equation $F_1(p,y)=0$ is identically satisfied.
The second equation yields:
\begin{equation}
  \eta(0)=\frac{1}{a \pi}\int_0^{2 \pi}cos(s)g(a cos(s))ds
\end{equation}

\subsection{Perturbation of the solution $x$ with respect to $y=(\eta, b)$ }
\begin{lemma} \label{lemm:gx2-gx1}
  Assume that $x$ is solution of \eqref{eq:xfeta} \eqref{eq:f=...g} with $f$ and $g$ Lipschitz with respect to all variables; 
  \begin{gather}
   | g(x_2)-g(x_1)| \le k|x_2-x_1|
  \end{gather}
  \begin{equation}
    \label{eq:f-lip}
    |f(x_2,\eta_2,\epsilon_2)-f(x_1,\eta_1,\epsilon_1)| \le k \left ( |x_2-x_1|+|\eta_2-\eta_1|+|\epsilon_2-\epsilon_1|   \right )
  \end{equation}
assume that the initial data are:
  \begin{equation}
    \label{eq:x0x'0}
    x(0)=a_{\alpha}, \;  x'(0)=\epsilon b_{\alpha}, \quad \left( \text{resp.} \; x'(0)= b_{\alpha} \right) \quad \text{ with } \alpha=1,2
  \end{equation}
  then, $x$ and $F$ are Lipschitz with respect to $y=(\eta,b)$ 
and $p$ with a modulus of magnitude $\epsilon k \bigl( resp.\;  k \bigr)  $ (where $k < +\infty$ ):
  \begin{gather}
    \label{eq:xlip-b-eta}
\forall  \theta \in [0,2 \pi], \;   \|x_2(\theta) - x_1(\theta)\| \le \epsilon k (|b_2-b_1|+|\eta_2-\eta_1|) \text{ and }\\
\Bigl [\text{ resp. when the initial velocity  is not of order } \epsilon \\
\forall  \theta \in [0,2 \pi], \;   \|x_2(\theta) - x_1(\theta)\| \le  k (|b_2-b_1|+|\eta_2-\eta_1|) \; \Bigr], \quad \text{ and }\\
\forall  \theta \in [0,2 \pi], \;|g(x_2(\theta))-g(x_1(\theta))| \le \epsilon k (|b_2-b_1|+|\eta_2-\eta_1|)  \label{eq:glip-b-eta} \\
\text{but we have only} \\
      |f(x(\theta,y_2),\eta_2,\epsilon)-f(x(\theta,y_1),\eta_1,\epsilon)| \le k \left ( \epsilon|b_2-b_1| + |\eta_2-h_1| \right) \\
    \|F(p,y_2)-F(p,y_1)\| \le  k (|b_2-b_1|+|\eta_2-\eta_1|)  \label{eq:Flip-b-eta}
  \end{gather}
\end{lemma}
\begin{proof}
  
\par{The proof} relies on the formula \eqref{eq:xintegrale}, we get:
\begin{gather}
  x_2(\theta)-x_1(\theta)=\epsilon(b_2-b_1)sin(\theta) -\epsilon \int_0^{\theta}
sin(\theta-s)\left [f(x_2,\eta_2,\epsilon) -f(x_1,\eta_1,\epsilon)
\right] ds \\
  |  x_2(\theta)-x_1(\theta) |\le\epsilon \left [ |b_2-b_1| +2\pi k|\eta_2-\eta_1| \right] + 2\pi \epsilon k|x_2-x_1|
\end{gather}
from which we get  equation of \eqref{eq:xlip-b-eta} and  we deduce from \eqref{eq:f-lip}
\begin{gather}
  |g(x(\theta,y_2)-g(\theta,y_1)| \le \epsilon k (|b_2-b_1|+|\eta_2-\eta_1|)
\end{gather}
which is \eqref{eq:glip-b-eta};but we get only
\begin{gather}
      |f(x(\theta,y_2),\eta_2,\epsilon)-f(x(\theta,y_1),\eta_1,\epsilon)| \le k \left ( \epsilon|b_2-b_1| + |\eta_2-h_1| \right)
\end{gather}
we deduce  equations  \eqref{eq:Flip-b-eta}.
\end{proof}
\subsection{A fixed point approximation method}
Consider  now the solution of \eqref{eq:mx"+g} , with $f$ given by \eqref{eq:f=...g}; when $g(a cos(s))$ is an even function, a $2 \pi$ periodic solution with $b=0$ satisfies trivially $F_1=0$; with lemma \ref{lemm:gx2-gx1}, we obtain:
\begin{proposition}
   For a a $2 \pi$ periodic solution with $b=0$ satisfies
   \begin{equation}
     \eta=(1-\epsilon \eta) \mathcal{F}(a,\epsilon,\eta)
   \end{equation}
with 
\begin{equation}
  \mathcal{F}(a,\epsilon,\eta)=\frac{\int_0^{2\pi} cos(s) g(x(s))ds}{\int_0^{2\pi} cos(s) x(s)ds}
\end{equation}
and for $\epsilon$ small enough,  this function satisfies
\begin{equation}
   |\mathcal{F}(a,\epsilon,\eta_2) - \mathcal{F}(a,\epsilon,\eta_1) \le  \epsilon k|\eta_2-\eta_1|
\end{equation}
so that the following sequence
\begin{gather}
      \eta^{k+1}=(1-\epsilon \eta^{k}) \mathcal{F}(a,\epsilon,\eta^{k}) 
\end{gather}
converges and it  proves the existence of a $2 \pi$ periodic solution of  \eqref{eq:xfeta} with \eqref{eq:f=...g}, hence the existence of a periodic solution of \eqref{eq:mx"+g}  with angular frequency \eqref{eq:xomega2}.

\end{proposition}

\subsection{Estimator problem}
We consider another approach which should be better conditioned (see Remark \ref{rem:natural}  ) and enable to consider arbitrary initial velocity whereas in the previous paragraph, the initial velocity is zero. 

 Equations \eqref{eq:F=0} may be considered as defining an implicit function
\begin{equation}
  \label{eq:p>y}
  p \longmapsto y
\end{equation}
As in previous subsection, this function $F$ is not smooth and we are going to show that the implicit function \eqref{eq:p>y} may be defined with a fixed point of a contraction maping following general lines of \cite{dontchev-rocka}.
So we introduce a  differential equation and a function $H$  that will be proved to be an estimator of $F$.
\begin{definition}
Consider the differential equation 
\begin{equation}
  \label{eq:ksiheta}
  \xi" + \xi +\epsilon h(\xi,\eta,\epsilon)=0
\end{equation}
with 
\begin{equation}
  \label{eq:h=}
  h(x,\eta,\epsilon)=-\eta \xi +(1-\epsilon \eta) \alpha \xi \quad \text{ with  } 0<\alpha <1
\end{equation}
and we introduce the following function:
\begin{align} \label{eq:H1=}
  H_1(p,y) &=\int_0^{2\pi}sin(s) h(\xi,\eta,\epsilon) ds \\
 \label{eq:H2=}  H_2(p,y) &=\int_0^{2\pi}cos(s) h(\xi,\eta,\epsilon) ds 
\end{align}
note that for $\epsilon=0$, $h=f$ so that $H=F$.
\end{definition}
\paragraph{Method:}
to solve $$F(p,y)=0$$ 
following the method introduced in paragraph \ref{subsec:quasi-lip},  we use the sequence $y^k$, defined by
\begin{equation}
  \label{eq:HEk+}
  H(p,y^{k+1})=-E(p,y^k), \text{ where } E(p,y)=F(p,y)-H(p,y)
\end{equation}
As $H$ is a smooth map, we prove that $H$ is invertible for $\epsilon$ small enough by using the classical implicit function theorem (see e.g. .

\begin{lemma}
\label{lem:jacoH}
  \begin{enumerate}
  \item 
  The Jacobi matrix of 
$$ y \longrightarrow H(p,y)$$ with $p=(a,\epsilon)$ and $y=(\eta,b)$ at $\epsilon=0$ and $b=0$ is
$$\begin{bmatrix}
0&  (-\eta(0) +\alpha) \pi \\
 -a \pi & 0
\end{bmatrix}$$
\item 
 so with $\alpha \neq \eta(0)=\frac{1}{2} $ and for $\epsilon$ small enough, the equation
\begin{equation}
  H(p,y)+e=0
\end{equation}
has a solution $y(p,e)$ which is differentiable and so locally Lipschitz with respect to its variables $p,e$ for $e$ small and $p$ close to $(a,0)$.
  \end{enumerate}
\end{lemma}
\begin{proof}
For $\epsilon=0$ and $b=0$, we notice that 
$$\frac{\p \xi}{\p \eta}=0, \quad \frac{\p \xi}{\p b}=sin(\theta),  $$ so
$$ \frac{\p h}{\p \eta}=-a cos(\theta), \quad \frac{\p h}{\p b}= (-\eta +\alpha)sin(\theta), \quad$$ from which we deduce the first part of the lemma. The second part is deduced from the classical implicit function theorem.
\end{proof}
\begin{remark}
  With examples \ref{subsub:examples}, when $f$ is defined by \eqref{eq:f=} and h by \eqref{eq:h=} then $f-h$ satisfies the general inequality  \eqref{eq:f-h.le.muxi}.
\end{remark}

\begin{lemma} \label{lemm:x-xi-le} 
    Assume that $x$ is solution of \eqref{eq:xfeta} \eqref{eq:f=...g} and $\xi$ solution of \eqref{eq:ksiheta}, \eqref{eq:h=}  with the same initial conditions and with $f$ and $g$ Lipschitz with respect to all variables and moreover that $h$ is a strict estimator of $f$, i.e.:
    \begin{equation}
\label{eq:f-h.le.muxi}
\text{ For small } \xi, \quad |f(\xi, \eta,\epsilon) -h(\xi, \eta,\epsilon)| \le \mu |\xi| 
    \end{equation}
then
\begin{equation}
  \label{eq:x-xi.le.epsmu}
  |x(\theta)-\xi(\theta)| \le \epsilon \mu \, c \, \underset{0 \le s \le 2\pi}{Sup}(|\xi|)
\end{equation}
for $\theta \in [0,2\pi]$
\end{lemma}

\begin{proof}  
Indeed, from \eqref{eq:xintegrale} with $x$ and a similar formula for $\xi$, both  having the same initial conditions, we get
\begin{gather}
  x(\theta)-\xi(\theta)=\epsilon \int_0^{\theta}sin(\theta -s) [ f(x,\eta,\epsilon) -h(\xi,\eta,\epsilon ) ] ds\\
\text{ and so }\\
  x(\theta)-\xi(\theta)=\epsilon \int_0^{\theta}sin(\theta -s) [ f(x,\eta,\epsilon)_ -f(\xi,\eta,\epsilon) +f(\xi,\eta,\epsilon) -h(\xi,\eta,\epsilon ) ] \\
\text{with \eqref{eq:f-h.le.muxi} and $f$ Lipschitz,} \quad | x(\theta)-\xi(\theta)|\le \epsilon k \int_0^{\theta} ( |x-\xi| + \mu |\xi| ) ds
\end{gather}
with Gronwall lemma (jl2), we get
\begin{gather}
   | x(\theta)-\xi(\theta)|\le \epsilon \mu  Sup(|\xi|) e^{2 \pi \epsilon k}
\end{gather}
for $\theta \in [0,2 \pi]$
\end{proof}

\begin{lemma}
\label{lem:lipE}
  Assume \eqref{eq:f-h.le.muxi}, then the function E(p,y)=F(p,y)-H(p,y), satisfies
  \begin{equation}
    \label{eq:lipE}
    lip(E,y) \le \bar{\epsilon} \mu < +\infty
  \end{equation}
with $p_0=(a, 0)^T$ with $a$ arbitrary, 
 in other words, $H$ is a strict estimator of $F$, uniformly in $p$  for $\epsilon \le \bar{\epsilon}$ (following  \cite{dontchev-rocka})
.
\end{lemma}
\begin{proof}
To prove it, we use the short hand $\tilde{f}=f(x(\theta,\tilde{y}),\tilde{\eta},\epsilon)$, $\tilde{h} =h(\xi(\theta,\tilde{y}),\tilde{\eta},\epsilon)$ where 
$\tilde{y}=(\tilde{\eta},\tilde{b})$ is close to $y=(\eta,b)$.
We have to estimate
\begin{gather}
E(p,\tilde{y})-E(p,y)=F(p,\tilde{y})-F(p,y) -(H(p,\tilde{y})-H(p,y) )
\end{gather}
where $F$ is defined in \eqref{eq:F1=},\eqref{eq:F2=} and $H$ in \eqref{eq:H1=},\eqref{eq:H2=},  so the basic point is to consider: 
 \begin{equation}
 \begin{split}
   \tilde{f}-f -(\tilde{h}-h)=&-\tilde{\eta} \tilde{x} +(1-\epsilon \tilde{\eta})g(\tilde{x})   \\
&-\left( -\eta x +(1-\epsilon \eta)g(x) \right )\\
&-\left( -\tilde{\eta} \tilde{\xi} +(1-\epsilon \tilde{\eta})g(\tilde{\xi}) \right)\\
&+\left( -\eta \xi +(1-\epsilon \eta)g(\xi) \right )
 \end{split}
\end{equation}
Then we split the right hand side into 3 expressions that we manipulate separately
\begin{equation}
\begin{split}
  |\eta(x-\xi)-\tilde\eta(\tilde x-\tilde\xi)|=&|\eta(x-\xi-\tilde x+\tilde \xi)+(\eta-\tilde\eta)(\tilde x - \tilde\xi)| \\
\le & \eta \bigl [|x-\tilde x|+|\tilde\xi -\xi| +|\tilde x-\tilde\xi||\eta-\tilde\eta| \bigr] \\
\text{with lemma \ref{lemm:gx2-gx1} and \ref{lemm:x-xi-le} } \\
\le & \epsilon c (1+Sup|\xi|)|\tilde y - y|
\end{split}
\end{equation}
Similarly
\begin{equation}
  \begin{split}
|  g(\tilde x) -g(x) -g(\tilde\xi)+g(\xi)| &\le k \left [ |\tilde x -x|+|\tilde\xi -\xi| \right ] \\
&\le \epsilon k |\tilde y - y| 
  \end{split}
\end{equation}
and

\begin{equation}
  \begin{split}
\epsilon &\left [ -\tilde\eta g(\tilde x) +\eta g(x) +\tilde\eta g(\tilde \xi) -\eta g(\xi) \right]=\\
\epsilon &  \left [ (\eta-\tilde\eta)g(x) +\tilde\eta(g(x)-g(\tilde x) ) - (\eta-\tilde\eta)g(\xi) -\tilde\eta(g(\xi)-g(\tilde \xi) \right ) ]\\
&\le \epsilon \left [ (Sup|g(x)|+Sup|g(\xi)| ) |\eta -\tilde\eta| 
+Sup|\tilde\eta|( |g(x)-g(\tilde x)|+|g(\xi)-g(\tilde \xi)|  ) \right ]  \\
&\le \epsilon c |\tilde y - y|
  \end{split}
\end{equation}
The last inequality is a consequence of  lemma \ref{lemm:gx2-gx1} and \ref{lemm:x-xi-le}

So we have obtained
\begin{equation}
  |\tilde f -f -(\tilde h - h) | \le \epsilon \mu |\tilde y - y |
\end{equation}
uniformly with respect to $p=(a,\epsilon)$ from which the lemma is obtained using the definitions of $ E(p,\tilde y)$ and $E(p,y)$
where $F$ is defined in \eqref{eq:F1=},\eqref{eq:F2=} and $H$ in \eqref{eq:H1=},\eqref{eq:H2=} .
\end{proof}

\begin{proposition}
\label{prop:periosol-1dof}
  Equation \eqref{eq:xfeta}, \eqref{eq:f=} (resp. \eqref{eq:f=...g}   ) has a $2 \pi$ periodic solution for $\epsilon$ small enough and it may be computed with the iterative process \eqref{eq:HEk+}; moreover, $\eta,b$ are Lipschitz function of $a, \epsilon$ and so
  \begin{gather} \label{eq:x-acos+erle}
    |x- a cos(\theta)-b sin(\theta) +\epsilon r(\theta, a, 0, \eta(0),0) | \le k \epsilon^2
  \end{gather}
\end{proposition}

\begin{proof}

The proof is now simple: write 
the iterative method  \eqref{eq:HEk+}
for $k$ and $k-1$
\begin{align}
  H(p,y^{k+1})=-E(p,y^k)\\
  H(p,y^{k})=-E(p,y^{k-1})
\end{align}
and by subtraction
\begin{equation}
   H(p,y^{k+1})- H(p,y^{k})= E(p,y^{k-1}) -E(p,y^k)
\end{equation}
with previous lemma \ref{lem:jacoH}, there exists $\lambda$
\begin{equation}
  \| y^{k+1} -y^k \| \le \lambda \| E(p,y^{k-1}) -E(p,y^k) \|
\end{equation}
and as $E$ is a strict estimator (lemma \ref{lem:lipE})
\begin{equation}
    \| y^{k+1} -y^k \| \le \epsilon \lambda \mu  \|y^k - y^{k-1} \|
\end{equation}
 As for the static case we should prove that for $p$ close to $p_0$ and $y^0$ close to the solution obtained for $\epsilon=0$, $y^1$ lies in a ball where $y(p,e)$ defined in lemma \ref{lem:jacoH}   is Lipschitz.

The first part of the proposition is then proved by the fixed point theorem.

The fact that $y$ is a Lipschitz function of $p$ is derived with a somewhat classical argument.
With lemma \ref{lem:jacoH}, we get from
\begin{equation}
  H(p_2,y_2)+E(p_2,y_2)=0, \quad  H(p_1,y_1)+E(p_1,y_1)=0
\end{equation}
\begin{equation}
  |y_2-y_1| \le \lambda |E(p_2,y_2) -E(p_1,y_1)|
\end{equation}
and with lemma \ref{lem:lipE}
\begin{equation}
  |y_2-y_1| \le \lambda (k|p_2-p_1| +\epsilon \mu |y_2-y_1|)
\end{equation}
from which
\begin{equation}
   |y_2-y_1| \le \frac{\lambda}{1-\epsilon \mu} (k|p_2-p_1| )
\end{equation}
Then, as  $r$ is Lipschitz with respect to $p$ and $y$, we obtain the inequality \eqref{eq:x-acos+erle}
\end{proof}

\section{Vibrations with several degrees of freedom}
\label{sec:vibrat-several}
\subsection{Differential equation model of a spring-mass system}

In the dynamic case, equation \eqref{eq:Ku+eps=Y} becomes:
\begin{equation}
  M \ddot{u} +B^T E B u + \epsilon B^T E'( Bu +d)_-= 0 \label{eq:dddotu+eps}
\end{equation}
with
\begin{equation}
  \dot{u}=\frac{du}{dt}
\end{equation}
\begin{remark}
 Many  other mechanical systems with unilateral properties can be cast into this frame in particular when using finite elements see e.g.\cite{jiang-pierre-shaw04}; the case of beams with support on a unilateral spring is considered in H Hazim \cite{hazim-these}.
As for one degree of freedom, we can consider a  \textbf{more general case:}
\begin{equation}
  M \ddot{u} +B^T E B u + \epsilon \mathcal{G}(u)= 0 \label{eq:dddotu+epsG}
\end{equation}
with $\mathcal{G}(u)$ a Lipschitz function; as an introduction we shall also consider the case of $\mathcal{G}(u)$ a smooth function; it can model a non linear stress-strain law.
\end{remark}
\subsection{Eigenvectors, estimator system}
In order to derive approximate periodic solutions, it is convenient to write the differential system in an eigenvector basis; we denote $K= B^T E B$ the rigidity matrix and introduce the matrix of generalized eigenvectors $\Phi=[\Phi_1, \dots, \Phi_j, \dots, \Phi_n]$
where $\Phi_k$ are generalized eigenvectors associated to generalized eigenvalues $\omega_k^2$ satisfying
\begin{equation}
\label{eq:omMphik}
  -\omega_k^2M\Phi_k + K \Phi_k=0
\end{equation}
or in matrix form:
$$ K \Phi = \Omega^2 M \Phi$$ with $\Omega^2$ the diagonal matrix of eigenvalues
; we assume $$\Phi^T M \Phi=I$$
consider the change of function 
$$u=\Phi \tilde x$$
we obtain:
\begin{equation}
\label{eq:x..Omega2}
   \ddot{\tilde x} +\Omega^2 \tilde x + \epsilon \Phi^T B^T E'( B\Phi \tilde x +d)_-=0
\end{equation}
and in \textbf{the general case}
\begin{equation}
\label{eq:x..Omega2G}
   \ddot{\tilde x} +\Omega^2 \tilde x + \epsilon \Phi^T \mathcal{G}(\Phi \tilde x)=0
\end{equation}

For unilateral spring, as $\gamma \longmapsto f(\gamma)=\gamma_-$ is a concave, Lipschitz function, we can use as strict estimator (in the sense of \cite{dontchev-rocka})
$\gamma \longmapsto h(\gamma)=\Lambda \gamma$ with $ \Lambda=diag(\lambda_j)$, $0 < \Lambda_j < 1$; if we set $e=f-h$, we have:
$$lip(e,0)=\max_j \max(\lambda_j,\lambda_j-1)$$
so a reasonable estimator differential system is
\begin{equation}
\label{eq:xi..Omega2}
   \ddot{\xi} +\Omega^2 \xi + \epsilon \Phi^T B^T E' \Lambda ( B\Phi \xi +d)=0
\end{equation}

\subsection{Periodic solutions}
In the linear case  when $\epsilon=0$, equation \eqref{eq:dddotu+eps} has periodic solutions of the form:
\begin{equation}
  \label{eq:harm-sol}
  u_k=(\upsilon_k e^{i \omega_k t}+ \overline{\upsilon_k } e^{-i \omega_k t} )  \Phi_k
\end{equation}
where $\Phi_k$ are generalized eigenvectors defined in \eqref{eq:omMphik}.

\subsubsection{Orientation}
We intend to prove that for $\epsilon$ small enough, there exists periodic solutions to equation \eqref{eq:dddotu+eps}, or in the general case \eqref{eq:dddotu+epsG}; the method is constructive  and will enable to derive a numerical scheme to approximate it; we use the equivalent system \eqref{eq:x..Omega2} or in the general case \eqref{eq:x..Omega2G} and  we introduce the following notations:
\begin{gather}
\label{eq:py=}
  p=(a_1,\epsilon)^T \quad y=(\eta,b_1,a_2,b_2, \dots, b_n)^T
\end{gather}
\begin{remark}
  Notice that $a_1$ plays a particular role: we are going to consider a periodic solution close to the harmonic solution of period $\frac{2 \pi}{\omega_1} $ of \eqref{eq:x..Omega2} with $\epsilon=0$; obviously, any other harmonic solution could be chosen by using a permutation of indexes.
\end{remark}

\subsubsection{For unilateral springs}

Following the lines of the 1 d.o.f. case, we introduce a change of variable suited to look for periodic solutions close to the eigenmode of index 1.
We set $\theta=\omega_{\epsilon} t$ with 

\begin{equation}
x(\theta)=\tilde x(t), \quad
\frac{1}{\omega_{\epsilon}^2 }=\frac{1-\epsilon \eta(\epsilon)}{\omega_1{^2}}
\quad \text{ and set } \quad x'=\frac{dx}{d\theta}
\end{equation}
Equation \eqref{eq:x..Omega2} becomes
\begin{gather}
  x"_1 + x_1 -\epsilon \eta x_1 +\epsilon\frac{1-\epsilon\eta}{\omega_1^2}\Phi_1^T B^T E' (B\Phi x +d)_-=0  \label{eq:x1"}\\
  x_j" + \frac{\omega_j^2}{\omega_1^2}x_j -\epsilon \eta\frac{\omega_j^2}{\omega_1^2} x_j +\epsilon\frac{1-\epsilon\eta}{\omega_1^2}\Phi_j^T B^T E'(B\Phi x +d)_-=0,   \quad j=2, \dots, n \label{eq:xk"}
\end{gather}

for the estimator differential system,
\begin{gather}
  \xi_1" + \xi_1 -\epsilon \eta \xi_1 +\frac{\epsilon(1-\epsilon\eta)}{\omega_1^2} \Phi_1^T B^T E
\Phi_1^T B^T \Lambda(B\Phi \xi +d)=0  \label{eq:xi1}\\
 \xi_j" +\frac{\omega_j^2}{\omega_1^2} \xi_j -\epsilon \eta\frac{\omega_j^2}{\omega_1^2} \xi_j + \frac{\epsilon(1-\epsilon \eta)}{\omega_1^2} \Phi_j^T B^T E \Lambda
(B\Phi \xi +d)=0, \quad j=2, \dots, n \label{eq:xik}
\end{gather}
we set:
\begin{gather}
f_1(x,\eta,\epsilon)=-\eta x_1 +\frac{1-\epsilon \eta}{\omega_1^2}  \Phi_1^T B^T E' (B \Phi x +d)_-  \label{eq:f1=}\\
f_j(x,\eta,\epsilon)= -\eta\frac{\omega_j^2}{\omega_1^2} x_j +\frac{1-\epsilon\eta}{\omega_1^2}\Phi_j^T B^T E'(B\Phi x +d)_-=0,   \quad j=2, \dots, n \label{eq:fk=}
\end{gather}
so that the system may be written:
\begin{gather}
  x"_1 +x_1 + \epsilon f_1(x,\eta,\epsilon) =0  \label{eq:x1"f1}\\
  x"_j +\frac{\omega_j^2}{\omega_1^2}x_j + \epsilon f_j(x,\eta,\epsilon) =0,   \quad j=2, \dots, n 
\label{eq:xk"fk}
\end{gather}
similarly, we set:

\begin{gather}
h_1(\xi,\eta,\epsilon)=-\eta \xi_1 +\frac{1-\epsilon \eta}{\omega_1^2}  \Phi_1^T B^T E' \Lambda(B \Phi \xi +d) \label{eq:h1=}\\
h_j(\xi,\eta,\epsilon)= -\eta\frac{\omega_j^2}{\omega_1^2}a \xi_j +\frac{1-\epsilon\eta}{\omega_1^2}\Phi_j^T B^T E'\Lambda(B\Phi \xi +d)=0,   \quad j=2, \dots, n \label{eq:hk=}
\end{gather}
and the ``estimator'' system is now:

\begin{gather}
  \xi_1" +\xi_1 + \epsilon h_1(\xi,\eta,\epsilon) =0  \label{eq:xi1"h1}\\
  \xi_j" +\frac{\omega_j^2}{\omega_1^2}\xi_j + \epsilon h_j(\xi,\eta,\epsilon) =0,   \quad j=2, \dots, n \label{eq:xik"hk}
\end{gather}

\subsubsection{General case}
In the general case of equation \eqref{eq:dddotu+epsG}, in the eigenvector basis 
\eqref{eq:x..Omega2G}, we still obtain equation \eqref{eq:x1"f1},\eqref{eq:xk"fk}, but with

\begin{gather}
f_1(x,\eta,\epsilon)=-\eta x_1 +\frac{1-\epsilon \eta}{\omega_1^2}  \Phi_1^T \mathcal{G}(\Phi x)  \label{eq:f1G=}\\
f_j(x,\eta,\epsilon)= -\eta\frac{\omega_j^2}{\omega_1^2} x_j +\frac{1-\epsilon\eta}{\omega_1^2}\Phi_j^T \mathcal{G} (\Phi x)=0,   \quad j=2, \dots, n \label{eq:fkG=}
\end{gather}
in the {\it non smooth case}, we consider a strict estimator $h_j(x)$ of $\Phi_j^T \mathcal{G} (\Phi x)$ and the estimator system is written as previously in \eqref{eq:xi1"h1},\eqref{eq:xik"hk}
\subsubsection{Condition of periodicity}

\begin{lemma}
\label{lem:xperio-r}
Consider the solution $\tilde x$ of the Cauchy problem of system \eqref{eq:x..Omega2} or \eqref{eq:x..Omega2G} and the solution $x$ of the associated system \eqref{eq:x1"f1}, \eqref{eq:xk"fk} after change of variable and with initial conditions:
\begin{align}
  x_1(0) -a_1&=0, \quad & \dot{x_1}(0)-b_1&=0 \\
 \vdots &\quad & \vdots \\
 x_n(0) -a_n&=0, \quad & \dot{x_n}(0)-b_n &=0 \\
\end{align}
For $\epsilon$ close to zero, $\tilde{x}$ is a solution of \eqref{eq:x..Omega2}
of period $\frac{2\pi}{\omega_{\epsilon}}$ with
$\frac{1}{\omega_{\epsilon}^2 }=\frac{1-\epsilon \eta(\epsilon)}{\omega_1{^2}}$
if and only if 
\begin{equation}
  F(p,y)=0
\end{equation}
where $(p,y)$ is defined in \eqref{eq:py=} and  the function 

\begin{equation}
  \label{eq:F}
  F:   \mathbb{R}^2 \times \mathbb{R}^{2n} \mapsto \mathbb{R}^{2n}
\end{equation}
is defined by:

\begin{align}
\label{eq:F1}
  F_1(p,y)= &\int_0^{2\pi}sin(s)f_1(x(s),\eta,\epsilon) ds \\
  F_2(p,y)= & \int_0^{2\pi}cos(s)f_1(x(s),\eta,\epsilon) ds\\
  F_{3}(p,y)= &x_2(2 \pi)-a_2  \\
\quad \vdots \\
  F_{2*n-1}(p,y)= &x_n(2 \pi)-a_n  \\
 F_{2*n}(p,y)= &\dot{x_n}(2 \pi)-b_n  \\
\label{eq:F2n}
\end{align}
\end{lemma}
\begin{proof}
  The proof is similar to the proof of lemma \ref{lem:rdrF1F2}

\end{proof}
In the non smooth case, it will be proved that  the following function $H$ is a strict estimator of $F$. 
\begin{definition}
  
Consider the solution $\xi$ of the Cauchy problem of system \eqref{eq:xi..Omega2} in the form \eqref{eq:xi1"h1}, \eqref{eq:xik"hk} with initial conditions:
\begin{align}
\label{eq:xi-init}
  \xi_1(0) -a_1 &= 0,  \quad & \dot{\xi_1}(0)-b_1&=0 \\
 \vdots &  \qquad  &\vdots & \\
 \xi_n(0) -a_n &= 0,  \quad & \dot{\xi_n}(0)-b_n&=0 \\
\end{align}
the function $H$ is defined by:
\begin{align}
\label{eq:Hdef1}
  H_1(p,y)=&\int_0^{2\pi}sin(s)h_1(\xi(s),\eta,\epsilon) ds \\
  H_2(p,y)= &\int_0^{2\pi}cos(s)h_1(\xi(s),\eta,\epsilon) ds\\
  H_{3}(p,y)= &\xi_2(2 \pi)-a_2  \\
\quad \vdots \\
  H_{2*n-1}(p,y)= &\xi_n(2 \pi)-a_n  \\
 H_{2*n}(p,y)= &\dot{\xi_n}(2 \pi)-b_n  \label{eq:Hdef2n}
\end{align}
\end{definition}

\subsection{The smooth case}
As indicated in the introduction, the analytic case is well known; here we just assume some differentiability; this type of result may be derived by other methods;we emphasize that this approach is {\it constructive} in particular if we approximate $x$ with a Fourier series (harmonic balance or spectral method).
\begin{lemma}
  In the case where $\mathcal{G}$ is smooth, the Jacobi matrix of
  \begin{equation}
    y \longmapsto F(p,y)
  \end{equation}
is invertible at $p_0=(a,0)$ with $b_1=0$; where $F$ is defined by
\eqref{eq:F1}, \eqref{eq:F2n}
\end{lemma}

\begin{proof}
  
We get for $\epsilon=0$
\begin{gather}
  x_1=a_1 cos(\theta)+ b_1 sin(\theta) \\
x_j=a_j cos(\frac{\omega_j}{\omega_1}\theta)+b_j\frac{\omega_1}{\omega_j} sin(\frac{\omega_j}{\omega_1}\theta)
\end{gather}
\begin{gather} 
 \frac{\p x_1}{\p \eta}=0,   \quad   \frac{\p x_1}{\p b_1}=sin(\theta)\\
 \frac{\p \dot{x_1}}{\p \eta}=0,   \quad \frac{\p \dot{x_1}}{\p b_1}=cos(\theta)
\end{gather}
so that for $b_1=0$,
\begin{gather}
  \frac{\p f_1}{\p \eta}=-a_1 cos(\theta) \qquad \frac{\p f_1}{\p b_1}=
-\eta sin(\theta) +\frac{1}{\omega_1^2} \Phi_1^T G'(\Phi x)  \Phi_1 sin(\theta)
\end{gather}
so that
\begin{gather}
\frac{\p F_1}{\p \eta}=0 \qquad  \frac{\p F_1}{\p b_1}=-\eta \pi +\frac{1}{\omega_1^2}\Phi_1^TG'(\Phi x)\Phi_1 \; \pi  \\
\frac{\p F_2}{\p\eta}=-a_1 \pi \qquad \frac{\p F_2}{\p b_1}=0 \quad  
\end{gather}
\begin{equation}
  det(\frac{\p(F_1,F_2)}{\p(\eta,b_1)}\neq 0
\end{equation}

For other indexes
\begin{gather}
  \frac{\p x_j}{\p a_j}=cos(\frac{\omega_j}{\omega_1}\theta), \quad \frac{\p x_j}{\p b_j}=\frac{\omega_1}{\omega_j}sin(\frac{\omega_j}{\omega_1}\theta)\\
 \frac{\p \dot{x_j}}{\p a_j}=-\frac{\omega_j}{\omega_1}\theta sin(\frac{\omega_j}{\omega_1}\theta), \quad \frac{\p \dot{x_j}}{\p b_j}=cos(\frac{\omega_j}{\omega_1}\theta)
\end{gather}
\begin{gather}
  \frac{\p F_{2j-1}}{\p a_j}=\frac{\p x_j}{\p a_j}(2\pi) -1, \quad \frac{\p F_{2j-1}}{\p b_j}= \frac{\p x_j}{\p b_j}(2\pi)\\
\frac{\p F_{2j}}{\p a_j}= \frac{\p \dot{x_j}}{\p a_j}(2\pi), \quad \frac{\p F_{2j}}{\p b_j}=\frac{\p \dot{x_j}}{\p b_j}-1
\end{gather}

\begin{gather}
  \frac{\p F_{2j-1}}{\p a_j}=cos(2\pi \frac{\omega_j}{\omega_1}) -1 \quad  \frac{\p F_{2j-1}}{\p b_j}= \frac{\omega_1}{\omega_j}sin(2\pi \frac{\omega_j}{\omega_1})  \\
\frac{\p F_{2j}}{\p a_j}= -\frac{\omega_j}{\omega_1} sin (2\pi \frac{\omega_j}{\omega_1})
 \quad \frac{\p F_{2j}}{\p b_j}=cos(2\pi \frac{\omega_j}{\omega_1}) -1
\end{gather}

When 
 $\frac{\omega_j}{\omega_1}$ is not an  integer for $j > 1$,

the determinant 
\begin{equation}
  \label{eq:df2j-1}
 det \frac{\p(F_{2*j-1},F_{2*j})}{ \p ( a_j'^{j+1},b_j'^{k+1} )} \neq 0
\end{equation}
ans so the Jacobi of  $F$ is invertible.
\end{proof}

\begin{proposition}
  Equation \eqref{eq:F1} to \eqref{eq:F2n} has a solution for $\epsilon$ small enough and this solution may be computed with an iterative method.
\end{proposition}
\begin{proof}
    We conclude the first part  from the classical implicit function theorem; the solution may then be approximated by a Newton method; it may be embedded in a continuation process.
\end{proof}

\subsection{Strict estimator of $F$}

\begin{lemma}
 The function function E(p,y)=F(p,y)-H(p,y), satisfies
  \begin{equation}
    \label{eq:lipÊ}
    lip(E,y) \le \epsilon \mu < +\infty
  \end{equation}
at $p_0=(a_1, 0)^T$ with $a_1$ arbitrary, 
(following the general lines of \cite{dontchev-rocka}) in other words, $H$ is a strict estimator of $F$ uniformly in $p$ close to $p_0$.
\end{lemma}

\subsubsection{Proof} 
The solution of \eqref{eq:x..Omega2} after change of variable $\theta=\omega_{\epsilon} t$, is solution of \eqref{eq:x1"} and \eqref{eq:xk"} or  \eqref{eq:x1"f1}, \eqref{eq:xk"fk};
it may be written 
\begin{align}
  x_1= &a_1 cos(\theta) + b_1 sin(\theta) -\epsilon \int_0^\theta sin((\theta-s)) f_1(x(s),\eta,\epsilon) \; ds  \label{eq:x1integrale=}\\
  x_j= &a_j cos(\frac{\omega_j}{\omega_1}\theta) + b_j \frac{\omega_1}{\omega_j} sin(\frac{\omega_j}{\omega_1}\theta) -\epsilon \frac{\omega_1}{\omega_j}\int_0^\theta sin(\frac{\omega_j}{\omega_1}(\theta-s)) f_j( x(s),\eta,\epsilon) \; ds \label{eq:xkintegrale=}
\end{align}

and the solution of \eqref{eq:xi..Omega2} after change of variable is
solution of \eqref{eq:xi1"h1} and \eqref{eq:xik"hk} with $h$ defined in \eqref{eq:h1=}, \eqref{eq:hk=}

\begin{align}
  \xi_1&= a_j cos(\theta) + b_j sin(\theta) -\epsilon \int_0^\theta sin(\theta-s) h_1(\xi(s),\eta,\epsilon) \; ds \\
 \xi_j &= a_j cos(\frac{\omega_j}{\omega_1}\theta) + b_j \frac{\omega_1}{\omega_j} sin(\frac{\omega_j}{\omega_1}\theta) -\epsilon \frac{\omega_1}{\omega_j}\int_0^{\theta} sin(\frac{\omega_j}{\omega_1}(\theta-s)) h_j(\xi(s),\eta,\epsilon) \; ds 
\end{align}
and so
 \begin{align}
   E_1(p,y)=&   \int_0^{2 \pi}  sin (s)  \left[f_1(x(s),\eta,\epsilon)- h_1(\xi(s),\eta,\epsilon) \right]\;ds\\
   E_2(p,y)=&   \int_0^{2 \pi}  cos (s)  \left[f_1(x(s),\eta,\epsilon)- h_1(\xi(s),\eta,\epsilon) \right]\;ds\\
   E_j(p,y)= &x_j(T)-\xi_j(T)= -\epsilon  \frac{\omega_1}{\omega_j} \int_0^{T}  sin( \frac{\omega_j}{\omega_1} (T-s))  \left[f_j(x(s),\eta,\epsilon)- h_j(\xi(s),\eta,\epsilon) \right]\;ds
 \end{align}

As the solution of a system of differential equations is Lipschitz with respect to the initial conditions and with respect to time, we obtain the lemma from  slight manipulations as in the proof of lemma \ref{lem:lipE}; absolute values are replaced by norms and we use the following lemmas. 

\subsubsection{Some lemmas}
\label{sec:lemmas-syst}

\begin{lemma} \label{lemm:gx2-gx1-syst}  
  Assume that $x$ is solution of \eqref{eq:x1"f1} \eqref{eq:xk"fk}   with $f$ and $g$ Lipschitz with respect to all variables; 
  \begin{gather}
   \| g(\check x )-g(x )\| \le k\|\check x -x \|
  \end{gather}
  \begin{equation}
    \label{eq:f-lip-syst}
    \|f(\check x ,\check \eta ,\check \epsilon )-f(x ,\eta ,\epsilon )\| \le k \left ( \|\check x -x \|+|\check \eta -\eta |+|\check \epsilon -\epsilon |  \right )
  \end{equation}

  then, $x$ and $F$ are Lipschitz with respect to $y=(\eta,b)$ 
and $p$; more precisely
assume that the initial data are:
  \begin{equation}
    \label{eq:x0x'0-syst}
    x(0)=a, \;  x'(0)=b, \quad ( \text{resp.} \; \check{x}(0)=\check{a}, \; \check{x}'(0)= \check b) \quad \text{ } 
  \end{equation}
then:
  \begin{gather}
    \label{eq:xlip-b-eta-syst}
\forall  \theta \in [0,2 \pi], \;   \|\check x(\theta) - x(\theta)\| \le  k (\|\check b-b\|+| \check\eta -\eta |) \text{ and }\\
\forall  \theta \in [0,2 \pi], \;\|g(\check x (\theta))-g(x (\theta))\| \le  k (\|\check  b -b \|+|\check \eta -\eta |)  \label{eq:glip-b-eta-syst} \\
\text{ and we have } \\
      \|f(x(\theta,\check y ),\check \eta ,\epsilon)-f(x(\theta,y ),\eta ,\epsilon) \| \le k \left ( \|\check b -b \| + |\check \eta -\eta | \right) \\
    \|F(p,\check y )-F(p,y )\| \le  k (\|\check b -b \|+|\check \eta -\eta |)  \label{eq:Flip-b-eta-syst}
  \end{gather}
\end{lemma}
\begin{proof}
  
\par{The proof} relies on the formula   \eqref{eq:x1integrale=} , we get:
\begin{gather}
 \check x_1 (\theta)-x_1(\theta)=(\check b_1 -b_1)sin(\theta) - \epsilon \int_0^{\theta}
sin(\theta-s)\left [f_1(\check x ,\check \eta ,\epsilon) -f_1(x,\eta,\epsilon)
\right] ds \\
  \| \check  x_1 (\theta)-x_1(\theta) \|\le \left [ \|\check b_1 -b_1\| +2\pi \epsilon k|\check \eta -\eta| \right] + 2\pi \epsilon k \underset{0\le s \le 2 \pi}{Sup} \|\check x(s) -x(s)\|
\end{gather}
and from \eqref{eq:xkintegrale=}, we get 
\begin{gather}
 \check x_j (\theta)-x_j(\theta)=\frac{\omega_1}{\omega_j} \Big [
(\check b_j -b_j)sin(\frac{\omega_j}{\omega_1}\theta) - \epsilon \int_0^{\theta}
sin(\frac{\omega_j}{\omega_1}(\theta-s))\left [f(\check x ,\check \eta ,\epsilon) -f(x,\eta,\epsilon)
\right] ds  \Big ]\\
  \| \check  x_j (\theta)-x_j(\theta) \|\le  \frac{\omega_1}{\omega_j}\left [ \|\check b_j -b_j\| +2\pi \epsilon j|\check \eta -\eta| + 2\pi \epsilon j \underset{0\le s \le 2 \pi}{Sup} \|\check x(s) -x(s) \| \right] 
\end{gather}
from which we get  equation of \eqref{eq:xlip-b-eta-syst} and  we deduce from \eqref{eq:f-lip-syst}
\begin{gather}
  \|g(x(\theta,\check y )-g(\theta,y)\| \le  k (\|\check b -b\|+|\check \eta -\eta|)
\end{gather}
which is \eqref{eq:glip-b-eta-syst}; we get 
\begin{gather}
      \|f(x(\theta,\check y ),\check  \eta ,\epsilon)-f(x(\theta,y),\eta,\epsilon)| \le k \left ( \|\check b -b\| + |\check \eta -h| \right)
\end{gather}
we deduce  equations  \eqref{eq:Flip-b-eta-syst}.
\end{proof}

\begin{lemma} \label{lemm:x-xi-le-syst} 
    Assume that $x$ is solution of \eqref{eq:x1"f1} \eqref{eq:xk"fk}   and $\xi$ solution of  \eqref{eq:xi1"h1}, \eqref{eq:xik"hk},   with the same initial conditions and with $f$ and $g$ Lipschitz with respect to all variables and moreover that $h$ is a strict estimator of $f$, i.e.:
    \begin{equation}
\label{eq:f-h.le.muxi-syst}
\text{ For small } \xi, \quad \|f(\xi, \eta,\epsilon) -h(\xi, \eta,\epsilon)\| \le \mu \|\xi \| 
    \end{equation}
then
\begin{equation}
  \label{eq:x-xi.le.epsmu-syst}
  \|x(\theta)-\xi(\theta)\| \le \epsilon \mu \, c \, \underset{0 \le s \le 2\pi}{Sup}(\|\xi\|)
\end{equation}
for $\theta \in [0,2\pi]$
\end{lemma}
\begin{proof}
  The proof is quite similar to the 1 d.o.f. case of lemma \ref{lemm:x-xi-le}
\end{proof}

\subsubsection{Implicit equation with the strict estimator $H$}
In order to compute and to prove the existence of periodic solutions, we consider a constructive approach as suggested in  the general lines of \cite{dontchev-rocka}.

\subsubsection{Iterative method}
\label{sec:iterative-method}

Solve 
\begin{gather}
  \label{eq:iterationHF}
 H(p,y^{k+1})=-E(p,y^k) \quad \text{ with } p=[a_1, \epsilon]^T, \; \\
 y^{k+1}=[\eta^{k+1}, b_1^{k+1},a_2^{k+1}, \dots ,b_n^{k+1} ]^T
\end{gather}

In other words, set $ x^k $ (resp. $\xi^k$)  the solution of \eqref{eq:x1"f1} \eqref{eq:xk"fk} (resp. \eqref{eq:xi1"h1}, \eqref{eq:xik"hk}) for the value $y=y^k$; similarly set 
$f^k=f(x^k,\eta^k,\epsilon)$, $h^k=h(\xi^k,\eta^k,\epsilon)$;
we have to solve \eqref{eq:xi..Omega2}  for $\epsilon$ small enough and with initial conditions
 \eqref{eq:xi-init} where the initial value $a_1$ is prescribed and others are to be found as well as the frequency parameter $\eta$  such that:
\begin{align}
\label{eq:iter1}
 \int_0^{2\pi} sin(s) h_1^{k+1} &=-\int_0^{2\pi} sin(s)(f_1^k- h_1^{k})ds\\
 \int_0^{2\pi} cos(s) h_1^{k+1} &=-\int_0^{2\pi} cos(s)(f_1^k- h_1^{k})ds\\
  \xi_2^{k+1}(2\pi)-a_2^{k+1} &=-[ x_2^k(2 \pi) -\xi_2^k(2 \pi)] ←  \\
 & \vdots \\
  \xi_n^{k+1}(2\pi)-a_n^{k+1} &=-[ x_n^k(2 \pi) -\xi_n^k(2 \pi)] ←  \\
  \dot{\xi}_n^{k+1}(2 \pi)-b_n^{k+1}& = -[\dot{x}_1^k(2 \pi) -\dot{\xi}_n^k(2 \pi)]  
\label{eq:iter2n}
\end{align}
As $H$ is a strict estimator of $F$, we may assume that for $\epsilon$ small enough, the right hand side is as small as needed; we denote it $\alpha, \beta$:
\begin{gather}
  \label{eq:alphabeta}
\alpha_1^k=\int_0^{2\pi} sin(s)(f_1^k- h_1^{k})ds, \quad  \beta_1^k=\int_0^{2\pi} cos(s)(f_1^k- h_1^{k})ds, \\
 \alpha_j^k=x_j^k(2 \pi) - \xi_j^k(2 \pi), \quad \beta_j^k=\dot{x}_j^k(2 \pi) -\dot{\xi}_j^k(2 \pi)
\end{gather}

 Equations \eqref{eq:iter1} to \eqref{eq:iter2n} with small right hand side denoted 
$\alpha, \beta$  are
\begin{gather}
 \int_0^{2\pi} sin(s) h_1^{k+1}=-\alpha_1, \quad  \int_0^{2\pi} cos(s) h_1^{k+1}=-\beta_2, \\
\xi_j^{k+1}(2 \pi)-a_j^{k+1}= - \alpha_j^k ,\quad 
\dot{\xi}_j^{k+1} (2 \pi)- b_j^{k+1}=  - b_j^k
\end{gather}

\begin{lemma}
  When the angular frequencies satisfy the property 
 $$\frac{\omega_j}{\omega_1}$$ is not an  integer for $j > 1$,

the jacobian matrix of 
  \begin{equation}
    y \longmapsto H(p,y)
  \end{equation}
defined in \eqref{eq:Hdef1} to \eqref{eq:Hdef2n} is invertible at $p_0=(a,0)$ with $b_1=0$
\end{lemma}
\begin{proof}
  
We get for $\epsilon=0$
\begin{gather}
  \xi_1=a_1 cos(\theta)+ b_1 sin(\theta) \\
\xi_j=a_j cos(\frac{\omega_j}{\omega_1}\theta)+b_j\frac{\omega_1}{\omega_j} sin(\frac{\omega_j}{\omega_1}\theta)
\end{gather}
\begin{gather} 
 \frac{\p \xi_1}{\p \eta}=0,   \quad   \frac{\p \xi_1}{\p b_1}=sin(\theta)\\
 \frac{\p \dot{\xi_1}}{\p \eta}=0,   \quad \frac{\p \dot{\xi_1}}{\p b_1}=cos(\theta)
\end{gather}
so that
\begin{gather}
  \frac{\p h_1}{\p \eta}=-a_1 cos(\theta) \qquad \frac{\p h_1}{\p b_1}=
-\eta sin(\theta) +\frac{1}{\omega_1^2} \Phi_1^TB^TE'\Lambda B \Phi [sin(\theta),0, \dots,0]^T 
\end{gather}
so that
\begin{gather}
\frac{\p H_1}{\p \eta}=0 \qquad  \frac{\p H_1}{\p b_1}=-\eta \pi +\frac{1}{\omega_1^2}\Phi_1^TB^T E'\Lambda B \Phi [ 1,0 ...0 ]^T \pi  \\
\frac{\p H_2}{\p\eta}=-a \pi \qquad \frac{\p H_2}{\p b_1}=0 \quad  
\end{gather}
\begin{equation}
  det(\frac{\p(H_1,H_2)}{\p(\eta,b_1)}\neq 0
\end{equation}
because for a suitable choice of $\Lambda$, we have  $\frac{\p h_1}{\p b_1} \neq 0$ ; the other derivatives are zero.
For other indexes
\begin{gather}
  \frac{\p \xi_j}{\p a_j}=cos(\frac{\omega_j}{\omega_1}\theta), \quad \frac{\p \xi_j}{\p b_j}=\frac{\omega_1}{\omega_j}sin(\frac{\omega_j}{\omega_1}\theta)\\
 \frac{\p \dot{\xi_j}}{\p a_j}=-\frac{\omega_j}{\omega_1}\theta sin(\frac{\omega_j}{\omega_1}\theta), \quad \frac{\p \dot{\xi_j}}{\p b_j}=cos(\frac{\omega_j}{\omega_1}\theta)
\end{gather}
\begin{gather}
  \frac{\p H_{2j-1}}{\p a_j}=\frac{\p \xi_j}{\p a_j}(2\pi) -1, \quad \frac{\p H_{2j-1}}{\p b_j}= \frac{\p \xi_j}{\p b_j}(2\pi)\\
\frac{\p H_{2j}}{\p a_j}= \frac{\p \dot{\xi_j}}{\p a_j}(2\pi), \quad \frac{\p H_{2j}}{\p b_j}=\frac{\p \dot{\xi_j}}{\p b_j}-1
\end{gather}

\begin{gather}
  \frac{\p H_{2j-1}}{\p a_j}=cos(2\pi \frac{\omega_j}{\omega_1}) -1 \quad  \frac{\p H_{2j-1}}{\p b_j}= \frac{\omega_1}{\omega_j}sin(2\pi \frac{\omega_j}{\omega_1})  \\
\frac{\p H_{2j}}{\p a_j}= -\frac{\omega_j}{\omega_1} sin (2\pi \frac{\omega_j}{\omega_1})
 \quad \frac{\p H_{2j}}{\p b_j}=cos(2\pi \frac{\omega_j}{\omega_1}) -1
\end{gather}

When 
 $\frac{\omega_j}{\omega_1}$ is not an  integer for $j > 1$,

the determinant 
\begin{equation}
  \label{eq:dh2j-1}
 det \frac{\p(H_{2*j-1},H_{2*j})}{ \p ( a_j'^{j+1},b_j'^{k+1} )} \neq 0
\end{equation}
ans so the Jacobian of  $H$ is invertible.
\end{proof}

By using the classical implicit theorem, we get
\begin{lemma}
  equation \eqref {eq:iterationHF}
or \eqref{eq:iter1} to \eqref{eq:iter2n} define a lipschitzian implicit function:
\begin{equation}
  p=[a_1, \epsilon]^T, [\alpha,\beta]^T \; \mapsto
 y^{k+1}=[T^{k+1}, b_1^{k+1},a_2^{k+1}, \dots ,b_n^{k+1} ]^T 
\end{equation}
\end{lemma}

\subsection{Computation and existence of periodic solutions}
\begin{proposition}
  Equation \eqref{eq:F1} to \eqref{eq:F2n} has a solution for $\epsilon$ small enough and this solution may be computed with the iterative method \eqref{eq:iterationHF}.
\end{proposition}
\begin{remark}
  The iterative method may be embeded in a continuation process by increasing the value of $\epsilon$
\end{remark}
\begin{proof}
  
The proof is now simple write 
the iterative method \eqref{eq:iterationHF} for $k$ and $k-1$
\begin{align}
  H(p,y^{k+1})=-E(p,y^k)\\
  H(p,y^{k})=-E(p,y^{k-1})
\end{align}
and by substraction
\begin{equation}
   H(p,y^{k+1})- H(p,y^{k})= E(p,y^{k-1}) -E(p,y^k)
\end{equation}
with previous lemma, there exists $\Lambda$
\begin{equation}
  \| y^{k+1} -y^k \| \le \Lambda \| E(p,y^{k-1}) -E(p,y^k) \|
\end{equation}
and as $E$ is a strict estimator (lemma \ref{eq:lipÊ})
\begin{equation}
    \| y^{k+1} -y^k \| \le \epsilon \Lambda \mu  \|y^k - y^{k-1} \|
\end{equation}
the proposition is proved
\end{proof}
\begin{proposition}
\label{prop:periosol-ndof}
  Equation 
\eqref{eq:x1"f1},\eqref{eq:xk"fk} with $f_1, \dots f_n$ given by   \eqref{eq:f1G=},\eqref{eq:fkG=}
 has a $2 \pi$ periodic solution for $\epsilon$ small enough and it may be computed with the iterative process \eqref{eq:HEk+}; moreover, $\eta,b$ are lipschitzian functions of $p=(a_1, \epsilon)$ 
\end{proposition}

\begin{proof}
The proof is similar to the one of Proposition \ref{prop:periosol-1dof}
\end{proof}
\begin{remark}
  This result should be compared with Proposition 2.1 of \cite{junca-br10}; in this paper we have an explcit approximate value of the frequency which gives an approximation of the solution  in an interval $[0,\gamma\epsilon^{-1}]$; here we get the existence of a period which gives an approximation of the solution on $[0, \infty]$; but if this period  is computed  numerically, the approximation of the solution will only remain valid on a finite interval to be determined. The result of \cite{junca-br10} may be viewed as a {\emph stability result}: if the frequency is computed up to the order $\epsilon^2$, the approximation of the solution remains valid in an interval $[0,\gamma\epsilon^{-1}]$.
\end{remark}

\section{Conclusion}

This approach seems to be the first one to provide a rigorous and constructive proof of existence of periodic solutions of non smooth systems of differential equations arising in structural mechanics.

Implementation of the algorithm with the harmonic balance principel is in progress.

It paves the way to reduced order modelling (see e.g. \cite{touze-amabili06}) for vibrating structures with non smooth non linearities such as contact with unilateral springs modeling bumpers.

\newcommand{\etalchar}[1]{$^{#1}$}
\def\cprime{$'$}

\end{document}